\documentclass[12pt,a4paper,enabledeprecatedfontcommands]{scrartcl}
\usepackage[utf8]{inputenc}
\usepackage[T1]{fontenc}
\usepackage{lmodern}

\usepackage[british]{babel}

\usepackage{doi,url}
\usepackage[numbers]{natbib}
\usepackage{etoolbox,letltxmacro,xifthen,ifdraft} 
\usepackage[dvipsnames]{xcolor}
\usepackage{amsmath,amssymb,enumitem,bbm,mathrsfs,amsthm,aliascnt,mathtools}
\usepackage{graphicx,subcaption}
\usepackage[retainorgcmds]{IEEEtrantools}
\usepackage{hyperref} 
\hypersetup{final,hidelinks}
\usepackage[atend]{bookmark} 

\providecommand{\email}[1]{\href{mailto:#1}{\nolinkurl{#1}}}

\setlist[enumerate,1]{label={(\roman*)}}
\setlist[enumerate,2]{label={(\alph*)}}
\setlist[enumerate,3]{label={(\Roman*)}}

\newcommand{\newsstheorem}[2]{
  \newaliascnt{#1}{dummy}
  \newtheorem{#1}[#1]{#2}
  \aliascntresetthe{#1}
  \expandafter\def\csname #1autorefname\endcsname{#2}
}
\numberwithin{dummy}{section}
\theoremstyle{plain}
  \newsstheorem{theorem}{Theorem}
  \newsstheorem{proposition}{Proposition}
  \newsstheorem{corollary}{Corollary}
  \newsstheorem{lemma}{Lemma}
  \newsstheorem{definition}{Definition}
  \theoremstyle{definition}
  \newsstheorem{example}{Example}
  \newsstheorem{assumption}{Assumption}
\theoremstyle{remark}
  \newsstheorem{remark}{Remark}

\newenvironment{eqnarr*}{\begin{IEEEeqnarray*}{rCl}}{\end{IEEEeqnarray*}\ignorespacesafterend}

\newcommand\RR{\mathbb{R}^d}

\newcommand\PP{\mathbb{P}}
\newcommand\QQ{\mathbb{Q}}
\newcommand\EE{\mathbb{E}}
\newcommand\NN{\mathbb{N}}
\newcommand\ZZ{\mathbb{Z}}

\newcommand\e{{\rm e}}

\newcommand\Ind{\mathbbm{1}}

\makeatletter \newcommand\mathof[1]{{\operator@font#1}} \makeatother

\newcommand\dd{\mathof{d}}
\newcommand\iu{\mathof{i}}

\DeclarePairedDelimiterX\ip[2]{\langle}{\rangle}{#1,#2}
\begin{document}

\title{Noise reinforcement \\for L\'evy processes}
\author{Jean Bertoin\footnote{Institute of Mathematics, University of Zurich, Switzerland}  }
\date{\unskip}
\maketitle 
\thispagestyle{empty}

\begin{abstract} {\bf Abstract:} 
In a step reinforced random walk, at each integer time and with a fixed probability $p\in(0,1)$, the walker  repeats one of his previous steps chosen uniformly at random, and with complementary probability $1-p$, the walker makes an independent new step with a given distribution. 
Examples in the literature include the so-called elephant random walk and the shark random swim. We consider here a continuous time analog, when the random walk is replaced by a L\'evy process. For sub-critical (or admissible) memory parameters $p<p_c$, where $p_c$ is related to the Blumenthal-Getoor index of the  L\'evy process, we construct a noise reinforced L\'evy process. Our main result shows that the step-reinforced random walks corresponding to discrete time skeletons of the L\'evy process, converge weakly to  the noise reinforced L\'evy process as the time-mesh goes to $0$. 
\end{abstract} 

\begin{abstract} {\bf R\'esum\'e:} 
Dans une marche al\'eatoire \`a pas renforc\'es, \`a chaque instant entier et avec une probabilit\'e fix\'ee $p\in(0,1)$, le marcheur r\'ep\`ete un de ses pr\'ec\'edents pas tir\'e uniform\'ement au hasard, et avec probabilit\'e $1-p$ effectue un nouveau pas ind\'ependant de loi donn\'ee.  Comme exemples dans la litt\'erature figurent l'elephant random walk et le shark random swim. Nous nous int\'eressons ici \`a un analogue en temps continu, c'est-\`a-dire lorsque la marche al\'eatoire est remplac\'ee par un processus de L\'evy. Pour des param\`etres de m\'emoire sous-critiques (ou encore admissibles) $p<p_c$, o\`u $p_c$ est li\'e \`a l'indice de Blumenthal-Getoor du processus de L\'evy, nous construisons un processus de L\'evy \`a bruit renforc\'e. Notre r\'esultat principal \'etablit la convergence en loi des marches al\'eatoires \`a pas renforc\'es associ\'ees aux squelettes discrets d'un processus de L\'evy, vers le processus de L\'evy \`a bruit renforc\'e, lorsque que le pas de la subdivision du temps tend vers $0$.
\end{abstract} 
{\small 
}

\section{Introduction}\label{s:intro}
The terminology random walk is commonly  used in the literature for two different families of stochastic processes in discrete time. First, a random walk can refer to a Markov chain on a graph, such that one-step transitions are allowed only between neighboring vertices.  Second, a random walk can also refer to the sequence  of the partial sums of i.i.d. random variables, say in $\RR$. Although  the main part of this work is concerned with the second meaning only,
we shall need in this introduction both interpretations, and it should be clear from the context which one is actually in use.

Informally, the {\em edge reinforced random walk}, which has  been introduced in 1986 by Coppersmith and Diaconis in a frequently cited unpublished article, is non-Markovian stochastic process on a locally finite graph, which
tends to travel ofter through edges it has already often travelled through in the past. 
More precisely, the conditional probability that
the $n$-th step of the process is from a vertex $x$ to a neighboring vertex $y$ given the trajectory up to time $n$,
is proportional to $a_e$ plus the number of times the edge $e=\{x,y\}$ has been previously crossed, where $a_e>0$ represents the initial weight of the edge $e$.  
This model and its variations (reinforcement can be non-linear, or rather concern vertices instead of edges, ...)  have triggered numerous works over the last decades.  Pemantle \cite{Pem} wrote a most useful survey some 12 years ago; we also refer to \cite{AngCrowKoz, CotThac, DisSabTar, MaUr, SabTar} and works cited therein for some more recent developments. 

The motivation for the present paper stems from a related notion of reinforcement, now for the second meaning of random walk. Informally, a  {\em step reinforced random walk}  tends to repeat steps that it has already often made in the past. 
Specifically, consider  a random walk $S(n)=X_1+\cdots + X_n$ where the steps  $X_1, X_2, \ldots$ form a sequence of i.i.d. variables in $\RR$. We update the increments $X_i$ one after another as follows.
Fix some $p\in(0,1)$, called the memory parameter, and let $(\varepsilon_i: i\geq 2)$ be an independent sequence of Bernoulli variables with parameter $p$.
We set first $\hat X_1=X_1$, and next for $i\geq 2$, we let $\hat X_i=X_i$ if $\varepsilon_i=0$, whereas    if $\varepsilon_i=1$, then we define $\hat X_i$ as a uniform random sample from $\hat X_1, \ldots, \hat X_{i-1}$.  
Finally, the sequence of the partial sums of the updated  sequence,  
$$\hat S(n)=\hat X_1+ \cdots + \hat X_n, \qquad n\in\NN,$$ is called a step reinforced random walk. In words, at each time, with probability given by the memory parameter $p$, $\hat S$ repeats one of its preceding steps chosen uniformly at random, and otherwise $\hat S$ has an independent increment with a fixed distribution. In the memoryless case $p=0$, one has $\hat S(n)=S(n)$, whereas 
$\hat S(n)=nX_1$ in the perfect memory case $p=1$. So, say when the steps are centered with finite variance, one naturally expects that the asymptotic behavior of the step reinforced random walk with memory $p\in(0,1)$ should somehow  interpolate between the diffusive behavior and the ballistic behavior with a random velocity. 

Step reinforcement
has been considered for the one-dimensional simple symmetric random walk, i.e. $d=1$ and  $\PP(X_i=1)=\PP(X_i=-1)=1/2$,   by K\"ursten \cite{Kur} as a version of the so-called {\em elephant random walk} \footnote{ Beware that most works on elephant random walks use a different -but actually  equivalent- construction, where the memory parameter corresponds to $(p+1)/2$ in the present notation.}. The latter is 
a model of a random walk with memory  which had been introduced previously by Sch\"utz and Trimper \cite{SchTr}
 and then studied by several authors, see e.g. \cite{BaurBer, Bercu, ColGavSch1, ColGavSch2}.
Quite recently, Businger \cite{Buart} investigated the scaling limits of the so-called {\em shark random swim}, a step reinforced random walk whose increments follow an isotropic stable distribution in $\RR$. 
A striking feature that has been pointed at in those works, is that the long-time behavior 
of $\hat S$ exhibits a phase transition at some critical memory parameter. Specifically, in the simple symmetric case, $n^{-1/2}\hat S(n)$ converges in distribution to some Gaussian law when $p<1/2$, whereas for $p>1/2$, $n^{-p}\hat S(n)$ converges almost surely to some 
(non Gaussian) random variable. If the steps follow an isotropic stable law with index $\alpha\in(0,2]$, then  $n^{-1/\alpha}\hat S(n)$ converges in distribution to some stable law when $\alpha p<1$, whereas $n^{-p}\hat S(n)$ converges almost surely to some (non zero) random variable when $\alpha p>1$.

The purpose of the present work is to  investigate  
the analog of step reinforcement in continuous time, that 
when the random walk $S=(S(n))_{n\in\NN}$ is replaced by a L\'evy process $\xi=(\xi(t))_{t\geq 0}$, for  memory parameters smaller than a certain critical value. Informally,  the steps 
$X_i$ of $S$ need to be substituted by the  increments of $\xi$ on infinitesimal time durations, that is by a L\'evy noise $\dd \xi(t)$.
We shall argue that the  existence of a stochastic process $\hat \xi$ whose time-derivative $\dd \hat\xi(t)$ in the sense of generalized function can be viewed as a reinforcement of the L\'evy noise $\dd \xi(t)$, depends crucially on the so-called Blumenthal-Getoor index $\beta$ of $\xi$. The latter was introduced in \cite{BG61}, and it is well-known that several local path properties of $\xi$ are then similar to those of a $\beta$-stable L\'evy process. We refer to the introduction in \cite{MansSchnurr} for some background and historical perspective on this notion. We shall see that 
$p_c=1/\beta$ is the critical memory parameter for noise reinforcement of L\'evy processes, specifically $\hat \xi$ can be defined when the memory parameter $p$ is {\em admissible}, that smaller than $1/\beta$, but this is no longer possible when $p\beta>1$. 
The most interesting situation is when $\beta>1$, as otherwise the inequality  $p\beta<1$ always holds. 
The L\'evy process $\xi$ has then a.s. bounded $\gamma$-variation for $\gamma>\beta$, and a.s. infinite $\gamma$-variation for $\gamma<\beta$. Informally, reinforcing the L\'evy noise $\dd \xi(t)$ with a super-critical memory parameter  $p>1/\beta$ disrupts the compensation mechanism at work for L\'evy processes with unbounded variation to the point that the reinforced noise cannot be integrated. 

An important general question about  reinforced processes is  to describe in which ways reinforcement alters the asymptotic behavior (recurrence, transience, trapping, asymptotic velocity, etc.; see e.g. \cite{ AngCrowKoz, CotThac, DisSabTar, Pem, SabTar} ).  Because of the repetition of certain jumps, a noise reinforced L\'evy process is usually not Markovian, and its distribution is singular with respect to that of the original L\'evy process. It is thus natural to ask similarly how noise reinforcement affects the local properties of the L\'evy process (oscillating behavior, rate of growth, etc.); however this will not be tackled here. 

The  rest of this work is organized in two main parts: in the first, we construct noise reinforced L\'evy processes,
and in the second, we consider random walks arising as discrete time skeletons of a L\'evy process and 
establish the weak convergence of the step reinforced random walks to the noise reinforced L\'evy process as the time-mesh goes to $0$.
More precisely, the plan is as follows.

Section 2.1  introduces Yule-Simon processes, a family of counting processes with one-dimensional marginal laws closely related to so called Yule-Simon distributions, and which form the building blocks for noise reinforcement. The material is elementary and doubtless already well-known, although I have often been unable to find precise references for some of the basic results which are discussed there. In Section 2.2, we investigate convergence of random sums of independent Yule-Simon processes, possibly after an adequate  compensation, and point at the key role of the Blumenthal-Getoor index in this framework.
This is used in Section 2.3 to construct noise reinforced L\'evy processes, essentially by mimicking the celebrated  L\'evy-It\^{o} decomposition; an important feature is that one can compute the multidimensional characteristic functions of such processes in terms of certain functionals of Yule-Simon processes. Our main result is claimed and established in Section 3: when the memory parameter $p$ is admissible, then the step reinforced random walk associated to a discrete time skeleton of a L\'evy process $\xi$ converges in distribution to the corresponding noise reinforced L\'evy process $\hat \xi$  as the time-mesh tends to $0$. The proof relies on the convergence of the empirical measure of the counting processes for occurrences in Simon's reinforcement model  \cite{Simon} towards the law of a Yule-Simon process, which requires techniques of propagation of chaos, and a uniform-integrability property of those counting processes. The fact that a phase transition occurs at the critical memory parameter $p_c=1/\beta$ is essentially a consequence of the property that Yule-Simon distributions are heavy-tailed.

\section{Construction of noise-reinforced L\'evy processes}
\subsection{Yule-Simon processes}\label{s:YSP}
Simon \cite{Simon} introduced a simple random dynamic for explaining the appearance of a family of heavy-tailed distributions in a variety of empirical data in biology, sociology, economics, ... It can be viewed as an early prototype of preferential attachment models later popularized by Barab\'asi and Albert, to which the step reinforced random walk is closely related. Typically, imagine one writes  a book as follows, producing a large random sequence of words. Fix a parameter $p\in(0,1)$, and write the first word of the book. Next, repeat this same word with probability $p$, or write a new word with probability $1-p$. When the length of the text is $k\geq 2$, the event that the $(k+1)$-th word of the text is a new one which has not appeared previously has probability $1-p$, and on the complementary event, the $(k+1)$-th word of the text is sampled uniformly at random among its $k$ first words. The book is completed when the text has reached a large length $n\gg 1$; one is then interested in the statistics of occurrences of words, 
and more precisely in the proportion of words which have been used exactly once,  twice, three times, ...  in the whole book. In this setting, Simon \cite{Simon}  pointed out the role of the probability mass function 
\begin{equation}\label{E:YSD}
\rho {\mathrm B}(k, \rho +1), \qquad k\geq 1,
\end{equation}
where ${\mathrm B}$ denotes the beta function and $\rho>0$ a parameter.  He  called \eqref{E:YSD} the Yule distribution, but nowadays \eqref{E:YSD} is rather referred to as the Yule-Simon distribution (with parameter $\rho$). In the present work, only the case $\rho=1/p>1$ arises,
although in the literature, $\rho$ may be any positive real number.

A well-known representation of the Yule-Simon distribution
as a mixture of geometric distributions driven by the exponential of the negative of an exponential variable 
 will be useful for us in the sequel. 
Namely, there is the identity
\begin{equation}\label{E:Mixture}\rho {\mathrm B}(k, \rho +1)=\rho \int_0^{\infty} \e^{-\rho t} (1-\e^{-t})^{k-1} \e^{-t} \dd t, \qquad k\geq 1.
\end{equation}

 We shall need an extension of Simon's result to processes, and in this direction, we introduce the following definition.  

\begin{definition}{(Yule-Simon process)} \label{D1} We call an integer-valued process on the unit time-interval, $Y=(Y(t))_{0\leq t \leq 1}$, a Yule-Simon process with parameter $\rho>0$, if $Y$ is 
 a time-inhomogeneous pure birth process started from $Y(0)=0$ a.s. and 
with time-dependent birth rates 
$$\lambda_k(t)\coloneqq \lim_{h \to 0+} h^{-1} \PP(Y(t+h)=k+1\mid Y(t)=k)\qquad \text{for $0\leq t<1$ and $k\in\NN$}$$
given by
 $$\lambda_0(t)=1/(1-t) \quad \text{and} \quad \lambda_k(t)=k/(\rho t) \text{ for }k\geq 1.$$

  \end{definition}
 In other words, a Yule-Simon process is a counting process\footnote{ This means   that $Y$ is a c\`adl\`ag integer-valued and non-decreasing process started from $0$ and having only jumps of unit length a.s.}  which fulfills the time-inhomogeneous Markov property and has infinitesimal generator at time $t$
$$ \lim_{h\to 0+} h^{-1}\EE[f(Y(t+h))-f(Y(t))\mid Y(t)=k] =\left\{ \begin{matrix}(f(1)-f(0))/(1-t) & \text{ for }k=0,\\
   (f(k+1)-f(k))k /(\rho t) & \text{ for }k\geq 1.\\
  \end{matrix}
  \right.
 $$
  We first point at the following simple construction in terms of a standard Yule process (i.e. a pure birth process with unit birth rate) taken in the logarithmic time.

  \begin{lemma} \label{L1} Let $Z=(Z(t))_{t\geq 0}$ be a standard Yule process 
  started from $Z(0)=1$ a.s., and $U$ an independent  uniform random variable on $[0,1]$.
  Set for every $t\in[0,1]$
  $$Y(t)= \left\{ \begin{matrix}0 & \text{ if }t < U,\\
  Z((\ln t - \ln U)/\rho) & \text{ if }t\geq U.\\
  \end{matrix}
  \right.
  $$
Then $Y=(Y(t))_{0\leq t \leq 1}$ is a Yule-Simon process with parameter $\rho$.   \end{lemma}
  \begin{proof} It is readily checked by time-substitution that $Y$ indeed satisfies the requirements of  Definition \ref{D1}.
    \end{proof}
   As an immediate consequence, we can now justify our terminology.
\begin{corollary} \label{C1}
Let $Y=(Y(t))_{0\leq t \leq 1}$ be a Yule-Simon process with parameter $\rho>0$. For every $0< t\leq 1$, we have:
  \begin{enumerate}
  \item $\PP(Y(t)\geq 1)=t$,
  \item the conditional law of $Y(t)$ given $Y(t)\geq 1$ is the Yule-Simon distribution \eqref{E:YSD} with parameter
  $\rho$. 
  \end{enumerate}
  \end{corollary}
 \begin{proof} Using the construction of Lemma \ref{L1}, we have $U=\inf\{t\geq0: Y(t)=1\}$, which yields the first assertion.
 Then conditionally on $Y(t)\geq 1$, that is equivalently $U\leq t$, $U/t$ has the uniform distribution on $[0,1]$ and thus $(\ln t - \ln U)/\rho$ has the exponential  distribution with parameter $\rho$. Recall that for every $r>0$, $Z(r)$ has the geometric law with parameter $\e^{-r}$, and that $Z$ is independent of $(\ln t - \ln U)/\rho$. The second assertion now follows from the representation \eqref{E:Mixture}  as a mixture of geometric distributions.
    \end{proof}
    \begin{remark} The argument of the proof shows more generally that 
   conditionally on $Y(t)\geq 1$,  the process  $(Y(s/t))_{0\leq s \leq 1}$ is again a Yule-Simon process with the same parameter.   
    \end{remark}
    
  We conclude this section by recalling 
  the asymptotic behavior 
\begin{equation}
\label{E:AsympYS}
{\mathrm B}(k, \rho +1) \sim \Gamma(\rho +1) k^{-(\rho+1)} \qquad \text{as }k\to\infty;
\end{equation}
This implies in particular that a Yule-Simon variable with parameter $\rho$ has finite moments of any order $r<\rho$, whereas its moment of order $\rho$ is infinite. For future use, we also record the following:
\begin{corollary}\label{L3} Let $Y=(Y(t))_{0\leq t \leq 1}$ be a Yule-Simon process with parameter $\rho$. Then:
\begin{enumerate}
\item If  $\rho>1$,  then for every $t\in[0,1]$
$$
\EE[Y(t)]=\frac{\rho}{\rho-1}\, t.
$$
\item If  $\rho>2$,  then for every $0\leq s \leq t \leq 1$
$$\EE[Y(t)Y(s)]=\frac{\rho^2}{(\rho-1)(\rho-2)} \,s \left(\frac{t}{s}\right)^{1/\rho}.
$$
\end{enumerate}
\end{corollary}
\begin{proof}
The first moment of the Yule-Simon law with parameter $\rho$ equals $\rho/(\rho-1)$ if $\rho>1$, 
and its variance equals $\rho^2(\rho-1)^{-2} (\rho-2)^{-1}$ if $\rho >2$. The formula (i) thus follows from Corollary \ref{C1}, and further we have for the second moment
 \begin{equation} \label{E:mean2YS}
\EE[Y(t)^2]= \frac{\rho^2}{(\rho-1)(\rho-2)}\, t\qquad \text{ if }\rho>2.
\end{equation}
On the other hand, a Yule process $Z$ grows exponentially in time and fulfills the 
branching property, so for every $k\geq 1$ and $0<s<t$, there is the identity 
$$\EE[Z(t)\mid Z(s)=k]= k\e^{t-s}.$$ Thanks to Lemma \ref{L1} and Corollary \ref{C1}, this yields
$$\EE[Y(t)\mid Y(s)=k]= k(t/s)^{1/\rho}.$$ 
The formula (ii) now follows from \eqref{E:mean2YS}.
 \end{proof}

\subsection{Random series of Yule-Simon processes}\label{s:RPI}

   Next we turn our attention to convergence of certain series $\sum_j Y_j x_j  $ of independent Yule-Simon processes, where the  $x_j$ are random vectors in $\RR $ that should be thought of as the sizes of the jumps made by some $d$-dimensional L\'evy process during the unit time interval.  
The  purpose of this section is to show that, under appropriate hypotheses and possibly after a proper compensation (i.e. centering), the sum of such a series does make sense. Although this is similar to the construction of L\'evy processes, we shall see that integrability properties of Yule-Simon variables play a crucial role.

 Specifically, we write $\QQ$ for the distribution of the Yule-Simon process with some fixed parameter $\rho>0$, say on the space of counting functions on the unit time interval.   Consider a Poisson point  measure ${\mathcal N}$ with intensity 
 $\nu\otimes \QQ$, where $\nu$ is a L\'evy measure on $\RR$, in the sense that 
 \begin{equation}\label{E:CLK}
 \int_{\RR} (1\wedge |x|^2)\nu(\dd x)<\infty, 
 \end{equation}
and  $|\cdot |$ stands for the Euclidean norm. The product form of this intensity measure allows us to assume that ${\mathcal N}$ is given in the form
  $${\mathcal N}=\sum_{j} \delta_{(x_j, Y_j)}$$
 with $\sum_{j} \delta_{x_j}$ a Poisson point measure on $\RR$ with intensity $\nu$, and $(Y_j)_{j\geq 1}$ an independent sequence of i.i.d. marks with law $\QQ$ (i.e. a sequence of i.i.d. Yule-Simon processes). 

We write  for $0\leq  a<b\leq \infty$
$$  {\mathcal N}_{a,b}\coloneqq \sum_{j} \Ind_{\{a\leq |x_j|< b\}}\delta_{(x_j, Y_j)},$$
which is again a Poisson point measure, now with  intensity
 $\nu_{a,b}\otimes \QQ$, where 
 $$\nu_{a,b}(\dd x) \coloneqq  \Ind_{\{a\leq |x|< b\}}\nu(\dd x).$$
Recall from the superposition property of Poisson point measures that for every $k\geq 1$ and $0\leq a_1<a_2< \ldots <a_k\leq \infty$,
the  Poisson point measures $  {\mathcal N}_{a_1, a_2},  {\mathcal N}_{a_2, a_3}, \ldots ,  {\mathcal N}_{a_{k-1}, a_k}$ are independent.

For every $0<a<b\leq \infty$,  $\nu_{a,b}$ is a finite measure and ${\mathcal N}_{a,b}$ possesses only finitely many atoms a.s. We can always define the step process
 \begin{equation}\label{E:defsig}
 \Sigma_{a,b}(t)\coloneqq \sum_{j} \Ind_{\{a\leq |x_j|<b\}} Y_j(t) x_j, \qquad t\in[0,1].
 \end{equation}

We assume henceforth that $\rho>1$, and  observe also from  Corollary \ref{L3}(i) and Campbell's formula that for every $a\in(0,1]$ and $0\leq t \leq 1$, there is the identity
$$\EE \left[ \sum_{j} \Ind_{\{a\leq |x_j|< 1\}} Y_j(t) |x_j|\right]= t \rho (\rho-1)^{-1}\int_{\{a\leq |x|<1\}} |x| \nu (\dd x).$$
So, if the L\'evy measure fulfills 
\begin{equation}\label{E:CVB}
\int_{\RR}(1\wedge |x|) \nu(\dd x)<\infty,
\end{equation}
which is a stronger requirement than \eqref{E:CLK},
then  $\Sigma_{0,1}(t)$ is a well-defined process given by a series that converges absolutely in $L^1(\PP)$,
and $\Sigma(t)\coloneqq \Sigma_{0,1}(t)+ \Sigma_{1,\infty}(t)$ is an a.s. absolutely convergent series.

In the general case when the L\'evy measure only fulfills \eqref{E:CLK}, we use again Campbell's formula for $a\in(0,1]$ to compute the mean vector
$$\EE \left[ \sum_{j} \Ind_{\{a\leq |x_j|< 1\}} Y_j(t) x_j\right]= t \rho (\rho-1)^{-1}\int_{\{a\leq |x|<1\}} x \nu (\dd x)$$
and 
define the compensated (or centered) sum
$$\Sigma^{\rm (c)}_{a,1}(t)\coloneqq \Sigma_{a,1}(t) -  t \rho (\rho-1)^{-1}\int_{\{a\leq |x|<1\}} x \nu (\dd x).$$ 
To investigate the behavior of this quantity as $a\to 0+$, it is  convenient to make the substitution $a=\e^{-r}$ with $r\geq 0$.
For every $t\in(0,1]$, the process $\left (\Sigma^{\rm (c)}_{\e^{-r},1}(t)\right)_{r\geq 0}$ has independent and centered increments, and is therefore a martingale. Whether or not this martingale converges as $r\to \infty$ depends crucially
on the so-called Blumenthal-Getoor (upper) index of the L\'evy measure $\nu$, which is defined as
\begin{equation}\label{E:BGind}
\beta(\nu)\coloneqq \inf\left \{b>0: \int_{|x|\leq 1} |x|^b\nu(\dd x) <\infty \right \}.
\end{equation}
In particular, we always have $\beta(\nu)\leq 2$, and if \eqref{E:CVB} fails, then $\beta(\nu)\geq 1$.  Further, in the stable case $\nu(\dd x) = |x|^{-\alpha-d} \dd x$ for some $0<\alpha<2$, one has plainly $\beta(\nu)=\alpha$. 
\begin{lemma}\label{L2} Suppose $\rho>1$ and take any $t\in(0,1]$.

\begin{enumerate} 
\item If $\beta(\nu)<\rho$, then the martingale $\left (\Sigma^{\rm (c)}_{\e^{-r},1}(t)\right)_{r\geq 0}$
is uniformly integrable, and in particular the compensated sum $\Sigma^{\rm (c)}_{a,1}(t)$
converges a.s. as $a\to 0+$. 

\item If $\beta(\nu)>\rho$, then a.s., $\Sigma^{\rm (c)}_{a,1}(t)$ does not converge as $a\to 0+$. 

\end{enumerate}

\end{lemma}

\begin{proof} For the sake of simplicity, we shall establish first parts (i) and (ii) under the assumption that the function $r\mapsto \nu(\{ |x|>\e^{-r}\})$ is continuous on $(0,\infty)$, i.e. that the 
L\'evy measure $\nu$ gives no mass to spheres. In parts (i') and (ii') of this proof, we shall briefly explain how this additional assumption can be removed.

(i) The martingale $\left (\Sigma^{\rm (c)}_{\e^{-r},1}(t)\right)_{r\geq 0}$ is purely discontinuous with quadratic variation process $\sum_{0< s \leq r} \Delta_s^2$, where $\Delta_s$ denotes the norm of the (possible) jump at $s$. So 
$\Delta_s= Y_j(t) |x_j|$ if the Poisson measure ${\mathcal N}$ has an atom at $(x_j, Y_j)$ with $|x_j|=\e^{-s}$ (our standing assumption ensures that a.s., ${\mathcal N}$ possesses at most one such atom for every $s>0$),
and $\Delta_s=0$ otherwise. 

Suppose first that $\beta(\nu)<2$ and pick any $q\in(\beta(\nu)\vee 1,\rho\wedge 2)$.
Recall from the Burkholder-Davis-Gundy inequality that   this martingale is bounded in $L^q(\PP)$ if and only if 
\begin{equation}\label{E:BDG}
\EE\left[ \left( \sum_{s\geq 0} \Delta_s^2\right)^{q/2}\right]<\infty.
\end{equation}
Since $q\leq 2$, there is basic inequality 
$$\left( \sum_{s\geq 0} \Delta_s^2\right)^{q/2} \leq \sum_{s\geq 0} |\Delta_s|^{q},$$
and we only need to check that
$$\EE\left[  \sum_{j} \Ind_{\{|x_j|<1\}}Y_j(t)^q |x_j|^q \right]<\infty.$$
In this direction, we 
get from Campbell's formula that 
$$\EE\left[  \sum_{j}  \Ind_{\{|x_j|<1\}}Y_j(t)^q |x_j|^q \right]= \EE[Y(t)^q] \int_{\{|x|<1\}} |x|^q \nu(\dd x).$$
Since $Y(t)\in L^q(\PP)$ (by \eqref{E:AsympYS} and Corollary \ref{C1}, as $q< \rho$) and 
$\int_{\{|x|<1\}} |x|^q \nu(\dd x)< \infty$ (from the definition of the Blumenthal-Getoor index, as  $q>\beta(\nu)$), our claim is established. 

Then suppose $\beta(\nu)=2<\rho$. Since $\int_{|x|\leq 1} x^2\nu(\dd x)<\infty$ anyway, we can perform the same calculation as above for $q=2$, and check that the martingale is bounded in $L^2(\PP)$. 

(ii) Recall from \eqref{E:AsympYS} and Corollary \ref{C1} that
$$ \PP(Y(t)>1/|x|)\sim t \Gamma(\rho+1) |x|^{\rho}\qquad \text{as } x\to 0.$$
Our assumption $\beta(\nu)>\rho$ implies that $\int_{|x|\leq 1} |x|^{\rho} \nu(\dd x)=\infty$, and it follows that
the Poisson point measure ${\mathcal N}$ possesses a.s. infinitely many atoms $(x,Y)$ with $Y(t)|x|>1$ (but of course only finitely many atoms with $|x|>\e^{-r}$). 
Hence the martingale $\left (\Sigma^{\rm (c)}_{\e^{-r},1}(t)\right)_{r\geq 0}$ has infinitely many jumps of length greater than $1$ a.s., and therefore  cannot converge.

(i') The argument in (i) does not directly apply when  the L\'evy measure $\nu$ can give a positive mass to some spheres, because  the Poisson measure ${\mathcal N}$ may now have several atoms on such spheres, and the quadratic variation process has then a less handy expression. However this is only a minor difficulty which can easily  be overcome, roughly speaking,  by stretching the parameter $r\geq 0$. Typically, if $\nu(\{|x|=\e^{-r}\})=m>0$, we slow down time, spending an amount of time  $m$ on the sphere with radius $\e^{-r}$,  and spread  the atoms of ${\mathcal N}$ that belong to that sphere uniformly  at random on the stretched time interval. Doing so, the original compensated martingale can be expressed in the form
$\Sigma^{\rm (c)}_{\e^{-s},1}(t)=M(s')$, where  $M$ is martingale with independent increments of the same type as in (i) (i.e. a compensated Poisson integral where the compensation is now continuous in the parameter $s'$) and $s'-s$ is the sum of the masses assigned by $\nu$ to spheres with radius between $\e^{-s}$ and $1$. The same calculation as in (i) can be applied to the martingale $M$ and show that the latter is uniformly integrable when $\beta(\nu)<\rho$.  {\em A fortiori}, the same holds for the original martingale $\Sigma^{\rm (c)}_{\e^{-s},1}(t)$, since it is merely a deterministic  time-change of the former.

(ii') When $\nu$ gives a positive mass to some spheres, the compensation for the Poissonian integral now yields deterministic jumps at the $r$'s for which $\nu(\{|x|=\e^{-r}\})>0$. However this induces no real difficulty and the argument in (ii) can easily be adapted. Specifically, we decompose $\nu=\nu' + \nu''$, where $\nu'$ is a L\'evy measure giving no mass to any sphere, and $\nu''$ a L\'evy measure supported by an at most countable union of spheres. If $\int_{|x|\leq 1} |x|^{\rho} \nu'(\dd x)=\infty$, then the argument of (ii) applies exactly the same. In the case when $\int_{|x|\leq 1} |x|^{\rho} \nu''(\dd x)=\infty$, we focus on the  $r$'s for which $\nu(\{|x|=\e^{-r}\})>0$ and check readily that the compensated martingale has again infinitely many jumps of length greater than $1$ a.s.
\end{proof}
\begin{remark} When the L\'evy measure is symmetric, the compensation term is always zero and $\Sigma_{a,1}(t)=
\Sigma^{(c)}_{a,1}(t)$ for all $a\in(0,1]$. Thus, if \eqref{E:CVB} fails and $\rho >\beta(\nu)$, 
then series $\sum Y_j(t)x_j$ is only semi-convergent, that is the limit $\lim_{a\to 0+} \sum \Ind_{\{a\leq |x_j| < 1\}}Y_j(t)x_j$ exists in $\RR$ although $\sum Y_j(t)|x_j|=\infty$ a.s.
\end{remark}

Throughout the rest of this work, we  write $\Sigma^{\rm (c)}_{0,1}(t)$ for the a.s. limit of $\Sigma^{\rm (c)}_{a,1}(t)$ as $a\to 0+$ whenever $\beta(\nu)<\rho$. 
We now compute the characteristic function of this process, and in this direction, we  write $\theta\cdot x$ for the scalar product
of two vectors $\theta$ and $x$ in $\RR$.

\begin{corollary}\label{C2} Let $\rho>1$ and $\nu$ be a L\'evy measure with Blumenthal-Getoor index
$\beta(\nu) < \rho$. Define for $\theta \in \RR$
$$\Phi_1(\theta)\coloneqq \int_{|x|\geq 1}(1-\e^{\iu \theta \cdot x})\nu(\dd x)
\ \text{ and } \ \Phi_0^{\rm (c)}(\theta)\coloneqq \int_{|x|< 1}(1-\e^{\iu \theta \cdot x}+\iu \theta \cdot x)\nu(\dd x).$$
The processes $\left(\Sigma_{1,\infty}(t)\right)_{0\leq t \leq 1}$
and $\left(\Sigma^{\rm (c)}_{0,1}(t)\right)_{0\leq t \leq 1}$ are independent. Further, for all $k\geq 1$, 
$\theta_1, \ldots, \theta_k\in\RR$
and  $t_1, \ldots, t_k\in[0,1]$, there are the identities 
$$\EE\left[\exp\left\{ \iu \sum_{j=1}^k \theta_j \cdot \Sigma_{1,\infty}(t_j)\right\} \right]=\exp\left\{ - \EE
\left[ \Phi_1\left(\sum_{j=1}^k Y(t_j)  \theta_j\right)\right]\right\}$$ 
and 
$$\EE\left[\exp\left\{ \iu \sum_{j=1}^k \theta_j \cdot \Sigma^{\rm (c)}_{0,1}(t_j)\right\} \right]=\exp\left\{ - \EE
\left[ \Phi_0^{\rm (c)}\left(\sum_{j=1}^k Y(t_j) \theta_j \right)\right]\right\}$$
where $Y=(Y(t))_{0\leq t \leq 1}$ denotes a Yule-Simon process with parameter $\rho$.
\end{corollary}
\begin{proof} The claim that $\Sigma_{1,\infty}$
and $\Sigma^{\rm (c)}_{0,1}$ are independent stems from the independence of the random measures
${\mathcal N}_{1,\infty}$ and  ${\mathcal N}_{0,1}$. The expression for the characteristic function of $\Sigma_{1,\infty}$
follows from the construction of the latter and the formula for the Fourier functional of Poisson point measures.
The same argument also yields  a similar formula for the characteristic function of the compensated sum $\Sigma^{\rm (c)}_{a,1}$ for every $0<a<1$, namely 
$$\EE\left[\exp\left\{ \iu \sum_{j=1}^k \theta_j \cdot \Sigma^{\rm (c)}_{a,1}(t_j)\right\} \right]=\exp\left\{ - \EE
\left[ \Phi_a^{\rm (c)}\left(\sum_{j=1}^k Y(t_j) \theta_j \right)\right]\right\}$$
with 
$$ \Phi_a^{\rm (c)}(\theta)\coloneqq \int_{a\leq |x|< 1}(1-\e^{\iu \theta \cdot x}+\iu \theta \cdot x)\nu(\dd x).$$

We then let $a\to 0+$, so  $\Phi_a^{\rm (c)}$ converges pointwise to  $\Phi_0^{\rm (c)}$.
On the other hand, pick any $q\in (\beta(\nu), \rho)$, and observe from the proof of Lemma 3.1 in \cite{BG61} that
$$\lim_{|\theta|\to \infty} |\theta|^{-q} \Phi_a^{\rm (c)}(\theta)=0, \qquad \text{uniformly in }a\in(0,1).$$
Recall that $Y(1)\in L^q(\PP)$; our last assertion now follows by letting $a\to 0+$ and applying Lebesgue dominated convergence. 
\end{proof}

\subsection{Noise reinforced L\'evy processes and their characteristic functions}
To start this section with, we consider the elephant random walk, that is the step reinforced random walk
$\hat S=(\hat S(n))_{n\in\NN}$ in the one-dimensional simple symmetric framework. It is known that
$\hat S$ has a diffusive behavior when the memory parameter $p<1/2$, 
and more precisely, it has been shown in \cite{BaurBer} that
the rescaled process $(n^{-1/2}\hat S(\lfloor nt\rfloor))_{t\geq 0}$ converges in distribution
to  the centered Gaussian process 
$\hat B=(\hat B(t))_{t\geq 0}$ with covariance 
\begin{equation}\label{E:coverw}
\EE\left[\hat B(s) \hat B(t)\right] =  \frac{t^ps^{1-p}}{1-2p} \qquad \text{for } 0\leq s \leq t.
\end{equation}
We  also refer to \cite{Bercu, ColGavSch1, ColGavSch2} for related works.

We sketch here a quick proof for this fact which may be of independent interest.
It is readily seen that the elephant random walk is a time-inhomogeneous Markov chain on $\ZZ$.
Specifically, the conditional probability
that its $(n+1)$-th increment is $\hat X_{n+1}=1$ given the trajectory up to time $n$, $(\hat S(j))_{0\leq j \leq n}$,  obviously equals 
$$(1-p)/2 + p(1/2+\hat S(n)/n)= 1/2 + p\hat S(n)/n,$$
and the conditional probability that $\hat X_{n+1}=-1$ equals $1/2 - p\hat S(n)/n$.
By classical diffusion-approximation techniques (see e.g. Section 7.4 in Ethier and Kurtz \cite{EtKu})
we deduce that $n^{-1/2}\hat S(\lfloor nt\rfloor)$ converges to the time-inhomogeneous diffusion
with infinitesimal generator at time $t$ given by ${\mathcal G}_tf(x)=\frac{1}{2} f''(x) + p x t^{-1}f'(x)$. We can also represent the limit process as the solution to the stochastic differential equation
$$\dd \hat B(t)= \dd B(t) + p t^{-1} \hat B(t) \dd t,$$
where $B=(B(t))_{t\geq 0}$ denotes  a standard linear Brownian motion.
One can solve this first-order linear differential equation explicitly and get
$$\hat B(t) = t^p\int_0^t s^{-p} \dd B(s), \qquad t\geq 0$$
(observe that the stochastic integral only makes sense for $p<1/2$), from which it is seen that
$\hat B$ is indeed a centered Gaussian process with covariance given by \eqref{E:coverw}.
We stress that this argument only works in this very situation, the Markov property fails for step reinforced random walks in higher dimension, or when the steps are not simple.

We call  a centered Gaussian process  with covariance  \eqref{E:coverw} a noise reinforced linear Brownian motion, and more generally in higher dimension $d\geq 1$, we call {\em noise reinforced Brownian motion} with memory parameter $p$ a process in $\RR$ whose coordinates are given by $d$ independent noise reinforced linear Brownian motions.
The following lemma points at another relation with Yule-Simon processes which will be important for us.

\begin{lemma}\label{L4} Let $M$ be a $d\times d$ matrix, and $\hat B$ a  noise reinforced Brownian motion in $\RR$ with memory parameter $p<1/2$. Then for all $k\geq 1$, 
$\theta_1, \ldots, \theta_k\in\RR$
and  $t_1, \ldots, t_k\in[0,1]$, we have
$$\EE\left[\exp\left\{ \iu \sum_{j=1}^k \theta_j \cdot M\hat B(t_j)\right\} \right]=\exp\left\{ -\frac{1-p}{2} \EE
\left[ {\mathbf q}\left(\sum_{j=1}^k Y(t_j)  \theta_j\right)\right]\right\},$$ 
where ${\mathbf q}$ denotes the positive semi-definite quadratic form on $\RR$ induced by $M$, i.e. 
${\mathbf q}(x)=|Mx|^2$.
\end{lemma}
\begin{proof} We first consider the case of dimension $d=1$, so the $\theta_i$ are reals and we may further assume that $M=1$.
We deduce from \eqref{E:coverw} that the variable 
$\sum_{j=1}^k \theta_j  \hat B(t_j)$ is Gaussian with covariance 
$$\frac{1}{1-2p} \sum_{j, \ell=1}^k \theta_j \theta_\ell (t_j\vee t_{\ell}) (t_j\wedge t_{\ell})^{1-p}.$$
Comparing with Corollary \ref{L3}(ii), we can express the above quantity as
$$(1-p) \EE
\left[\left(\sum_{j=1}^k Y(t_j)  \theta_j\right)^2\right],$$ 
and hence 
$$\EE\left[\exp\left\{ \iu \sum_{j=1}^k \theta_j \hat B(t_j)\right\} \right]=\exp\left\{ -\frac{1-p}{2} \EE
\left[\left(\sum_{j=1}^k Y(t_j)  \theta_j\right)^2\right]\right\}. 
$$

In dimension $d\geq 1$,  we then obtain the formula of the statement when $M={\rm Id}_d$ is the identity matrix
and thus ${\mathbf q}$ is the square of the Euclidean norm, by using the independence of the coordinates. 
Finally, the case when $M$ is an arbitrary $d\times d$ matrix follows, writing
$\theta_j \cdot M\hat B(t_j)= M^T\theta_j \cdot \hat B(t_j)$.
\end{proof}

We next briefly recall the L\'evy-It\^{o} construction of L\'evy processes in $\RR$.
Consider a positive semi-definite quadratic form ${\mathbf q}$ on $\RR$, $a\in\RR$, and $\Lambda$ a L\'evy measure on $\RR$. 
Let $M$ be any $d\times d$ matrix such that $|M x|^2={\mathbf q}(x)$, $B$ a $d$-dimensional Brownian motion, 
and ${\mathcal M}(\dd t, \dd x)$ a Poisson random measure on $[0,\infty)\times \RR$ with intensity $\dd t \times \Lambda(\dd x)$ which is independent of $B$. 
Then,  the process $\xi=(\xi(t))_{t\geq 0}$ defined by
$$\xi(t)= M B(t) + ta+ \int_{[0,t]\times \RR} x \Ind_{\{|x|\geq 1\}}{\mathcal M}(\dd s, \dd x)
+ \int_{[0,t]\times \RR} x \Ind_{\{|x|< 1\}}{\mathcal M}^{\rm (c)}(\dd s, \dd x),$$
where the second stochastic integral with the notation ${\mathcal M}^{\rm (c)}$ refers to the compensated version of Poissonian integral, is a L\'evy process  with characteristics $({\mathbf q},a,\Lambda)$.
The characteristic exponent of $\xi$ is given by the L\'evy-Khintchin formula
\begin{equation}\label{E:LKF}
\Psi(\theta)\coloneqq \frac{1}{2} {\mathbf q}(\theta) - \iu a\cdot \theta + \int_{\RR} (1-\e^{\iu \theta\cdot x} + \iu \theta\cdot x \Ind_{\{|x|<1\}}) \Lambda(\dd x)\,, \qquad \theta \in \RR,
\end{equation}
and  we have
$\EE\left[\e^{ \iu  \theta \cdot  \xi(t)} \right]=\exp\left\{ -t \ {\Psi}( \theta)\right\}$.
 Further, the symmetric bilinear form associated to ${\mathbf q}$ is the covariance matrix of the Gaussian component $MB$ of $\xi$; recall also that the sample paths of $\xi$ have finite variations a.s. if and only if ${\mathbf q}=0$ and $\int_{|x|\leq 1}|x|\Lambda (\dd x)<\infty$.

We will now define the noise reinforced version $\hat \xi$ of $\xi$. In this direction,  recall from the preceding section
that $\QQ$ stands for the law of a Yule-Simon process with parameter $\rho=1/p$, and the notation for random series of Yule-Simon processes. Recall also from  \eqref{E:BGind} that  $\beta(\nu)$ stands for the Blumenthal-Getoor index of a L\'evy measure $\nu$, and define the Blumenthal-Getoor index of  a L\'evy process $\xi$ with characteristics $({\mathbf q},a,\Lambda)$ as 
$$\beta=\left\{ \begin{matrix} &\beta(\Lambda)& \text{ if }{\mathbf q}=0, \\
&2& \text{ if }{\mathbf q}\neq 0.
\end{matrix} \right.
$$
We then say that $p\in(0,1)$ is an {\em admissible memory parameter} for $({\mathbf q},a,\Lambda)$ if $p\beta<1$.

\begin{definition} [Noise reinforced L\'evy process] \label{D2}
Let ${\mathbf q}$ be a positive semi-definite quadratic form on $\RR$, $a\in\RR$, $\Lambda$ a L\'evy measure on $\RR$
and $p\in(0,1\wedge \beta^{-1})$ an  admissible memory parameter for $({\mathbf q},a,\Lambda)$.
Let ${\mathcal N}=\sum \delta_{(x_j,Y_j)}$ be a Poisson point measure with intensity $\nu \otimes \QQ$, where $\nu=(1-p)\Lambda$. In the case ${\mathbf q}\neq 0$, let further $M$ be any $d\times d$ matrix such that  $|M x|^2={\mathbf q}(x)$
and $\hat B$  a $d$-dimensional noise reinforced Brownian motion with memory parameter $p$, which is independent of ${\mathcal N}$. 

We now define a process $\hat \xi=(\hat \xi(t))_{0\leq t \leq 1}$ by
$$\hat \xi(t) \coloneqq M \hat B(t) + ta+\Sigma_{1,\infty}(t) + \Sigma^{\rm (c)}_{0,1}(t),$$
where in the case ${\mathbf q}=0$,  the term $M \hat B(t)$ should be interpreted as $0$.
We call $\hat \xi$ a
{\em noise reinforced L\'evy process} with characteristics $({\mathbf q},a,\Lambda,p)$.\end{definition}

It may be worthwhile to comment a bit on this definition. The construction bears obvious similarities with the L\'evy-It\^{o} decomposition, replacing Brownian motion by its reinforced version, and the role of Poisson processes being played by Yule-Simon processes. The sizes of the jumps  made by the L\'evy process $\xi$ during
the unit time interval form a Poisson point process with intensity $\Lambda$, and roughly speaking, the effect of the reinforcement with memory parameter $p$ implies that each such jump is erased with probability $p$, independently one from the other. This corresponds to a $(1-p)$-thinning, and hence, if we consider the family of jumps of the reinforced process and disregard their multiplicities (which correspond to repetitions of the same jump), one gets a Poisson point measure with intensity $\nu=(1-p)\Lambda$. Plainly, step reinforcement has no effects on a pure drift process $t\mapsto at$, which explains why, contrary to the L\'evy measure, the drift coefficient remains unchanged. 
Note also that in general, the repetition of certain jumps not only impedes the Markov property of noise reinforced L\'evy processes, but also entails that the distribution of $\hat \xi$ is singular with respect to that of $\xi$. 

We now arrive at the following formula for the  characteristic function of a noise reinforced L\'evy process.

\begin{corollary} \label{C3} Let $\hat \xi$ be a 
 noise reinforced L\'evy process with characteristics $({\mathbf q},a,\Lambda,p)$.
Then for all $k\geq 1$, 
$\theta_1, \ldots, \theta_k\in\RR$
and  $t_1, \ldots, t_k\in[0,1]$, we have
$$\EE\left[\exp\left\{ \iu \sum_{j=1}^k \theta_j \cdot \hat \xi(t_j)\right\} \right]=\exp\left\{ -(1-p) \EE
\left[ {\Psi}\left(\sum_{j=1}^k Y(t_j)  \theta_j\right)\right]\right\},$$ 
where $\Psi$ denotes the characteristic exponent \eqref{E:LKF} of a L\'evy process $\xi$ with characteristics $({\mathbf q},a,\Lambda)$, and $Y=(Y(t))_{0\leq t \leq 1}$ a Yule-Simon process with parameter $\rho=1/p$. 
\end{corollary}
\begin{proof}
Since $\nu=(1-p)\Lambda$, in the notation of Corollary \ref{C2}, there is the identity
$$(1-p)\Psi(\theta)=\frac{1-p}{2} {\mathbf q} (\theta)- \iu (1-p)a\cdot \theta+ \Phi_1(\theta)+\Phi_0^{\rm (c)}(\theta), \qquad \theta\in\RR.$$ Our claim is now plain from the construction of $\hat \xi$,  by combining Corollary \ref{C2}, Lemma \ref{L4} and Corollary \ref{L3}(i). 
\end{proof}
For instance, we deduce from Corollary \ref{L3}(i) and Corollary \ref{C3} that if $\xi$ is a standard Cauchy process, i.e. with  characteristic exponent $\Psi(\theta)=|\theta|$, then every memory parameter $p\in(0,1)$ is admissible and for every $t\in(0,1]$,  $\hat \xi(t)$ has the Cauchy distribution with scale parameter $t$
(and thus $\hat \xi$ has the same one-dimensional marginals as $\xi$). 

To conclude this section, we point at a few simple properties 
that noise reinforced L\'evy processes inherit from usual L\'evy processes, and which can be seen either directly from the construction in Definition \ref{D2}, or  easily checked from Corollary \ref{C3} (and hence proofs are omitted). 

\begin{proposition}\label{P3} 
\begin{enumerate}
\item {\rm [Additivity property]} Let $\hat \xi$ and $\hat \xi'$ 
be two independent noise reinforced L\'evy processes with respective characteristics $({\mathbf q},a,\Lambda,p)$
and $({\mathbf q}',a',\Lambda',p)$. The sum $\hat \xi + \hat \xi'$
is then a noise reinforced L\'evy process with characteristics $({\mathbf q}+{\mathbf q}',a+a',\Lambda+\Lambda',p)$.

\item{\rm [Independence of coordinates]} 
The coordinates $(\hat \xi_1, \ldots, \hat \xi_d)$ of a $d$-dimensional noise reinforced L\'evy process $\hat \xi$ are independent if and only if the characteristic exponent $\Psi$ has  the form
$$\Psi(\theta)=\Psi_1(\theta_1)+\cdots + \Psi_d(\theta_d), \qquad \theta=(\theta_1, \ldots, \theta_d)\in\RR,$$
that is equivalently, if and only if the coordinates $(\xi_1, \ldots,  \xi_d)$ of the $d$-dimensional L\'evy process $\xi$ are independent.

\item{\rm [Stability]} Suppose that the L\'evy process $\xi$ is (strictly) stable with index $\alpha \in(0,2]$,
that is its characteristic exponent  fulfills $\Psi(c\theta)= c^{\alpha} \Psi(\theta)$ for all $\theta \in \RR$ and $c>0$. Then  the memory parameter $p\in(0,1)$ is admissible if and only if $p\alpha <1$, and in that case, the noise reinforced version $\hat \xi$ is an $\alpha$-stable process, in the sense that its finite dimensional marginals follow $\alpha$-stable laws. 
\end{enumerate}
\end{proposition}

\begin{remark}
Statements (i) and (ii)  might come as a surprise, as in general, the sum of two independent step-reinforced random walks is not a step-reinforced random walk, and  likewise, step reinforcement does not preserve the independence of coordinates for random walks.
\end{remark} 
\section{Noise reinforcement as limit of step reinforcements}

We now have all the ingredients needed to state the main result of this work. Recall from the introduction the definition of the step reinforced version $\hat S$ of a random walk $S$ for a given memory parameter $p\in(0,1)$. 

\begin{theorem}\label{T1}
Let  $\xi=(\xi(t))_{t\geq 0}$ be a L\'evy process with characteristics $({\mathbf q},a,\Lambda)$. For every $n\geq 1$, we write $S^{(n)}=(\xi(k/n))_{k\geq 0}$ for  the discrete time skeleton of $\xi$ with time-mesh $1/n$.
Let $p\in(0,1)$ be an admissible memory parameter for $({\mathbf q},a,\Lambda)$ and let $\hat S^{(n)}=(\hat S^{(n)}(k))_{k\geq 0}$ denote the step reinforced random walk with memory parameter $p$. 

Then as $n\to \infty$, there is the weak convergence in the sense of finite dimensional distributions
$$(\hat S^{(n)}(\lfloor tn\rfloor))_{0\leq t \leq 1} \ \Longrightarrow \ (\hat \xi(t))_{0\leq t \leq 1}$$
where $\hat \xi$ denotes a noise reinforced L\'evy process with characteristics $({\mathbf q},a,\Lambda,p)$.
\end{theorem}

\begin{remark}
In the case when $\xi$ is an isotropic stable L\'evy process, Theorem \ref{T1} bears a close relation to Theorem 2 in \cite{Buart} and its multi-dimensional extension in \cite{BuarX}. Specifically, Businger established  the weak convergence in the sense of finite dimensional distributions of $\hat S^{(n)}(\lfloor \cdot n\rfloor)$ towards some non-L\'evy stable process that is  specified via its characteristic function; a representation as an integral an $\alpha$-stable random measure is also given. Here, we further identify this limit a noise reinforced (stable) L\'evy process  $\hat \xi$. Although the present approach also relies on establishing the convergence of characteristic functions, our proof based on propagation of chaos techniques differs much from the tedious estimates in \cite{Buart} and \cite{BuarX} which are specific to the stable case.  
\end{remark}

The rest of this section is devoted to the proof of Theorem \ref{T1}, and we shall first connect asymptotically step reinforcement 
and Yule-Simon processes.

Recall that $X_1, X_2, \ldots$ is an i.i.d. sequence of random variables, and $\varepsilon_2, \varepsilon_3, \ldots$  an independent i.i.d. sequence of Bernoulli variables with parameter $p$.  It will be  convenient to pretend that all the variables $X_1, X_2, \ldots$ take distinct values\footnote{ If this is not the case, then we can simply add a mark $k$ to each variable $X_k$ for the sole purpose of distinguishing between the values of variables with different indices, and then ignore marks when taking partial sums to define the step reinforced random walk.}. 
Then viewing $X_1$ as the first word, and for $k\geq 2$, each variable $X_k$ for which $\varepsilon_k=0$ as a new word that never occurred previously, the sequence $\hat X_1, \hat X_2, \ldots$ of the increments of the step reinforced random walk that was described in the introduction follows precisely Simon's model that has been recalled in Section 2.
We introduce the counting processes for occurrences, namely we set
for all integers $j,k\geq 1$
$$N_j(k)\coloneqq {\rm Card}\{1\leq \ell \leq k: \hat X_{\ell} = X_j \}, $$
so that the step reinforced random walk can be expressed in the form 
\begin{equation}\label{E:SwithN}
\hat S(k)=\sum_{j=1}^{\infty} N_j(k) X_j .
\end{equation}
Note that   when $\varepsilon_j=1$,  an event which has probability $p$, $X_j$ will never be used as a word and 
 $N_j(k)=0$ for all $k\geq 1$,  whereas  when $\varepsilon_j=0$, one has 
$N_j(k)=0$ for $k< j$ and $N_j(k)\geq 1$ for $k\geq j$.

   The main purpose of this section is to point out that the Yule-Simon process arises in the limit of the empirical distribution of (time-rescaled versions of) the counting processes. In this direction, we consider a complex valued functional $F$ on the space of c\`adl\`ag paths $\omega: [0,1] \to \NN$, which is continuous for Skorohod's topology. For simplicity, we further assume that $F(0)=0$, where by a slightly abusive notation, we denote  the path identically zero by $0$.

\begin{proposition}\label{P1} If $F$ is bounded, then there is the convergence in probability
$$\lim_{n\to \infty} \frac{1}{n}\sum_{j= 1}^{n} F\left( N_j(\lfloor \cdot n\rfloor)   \right)= (1-p) \EE\left[ F(Y)\right],$$
where we use the notation $ N_j(\lfloor \cdot n\rfloor)$ for the time-rescaled counting process $\left(N_j(\lfloor t n\rfloor)\right)_{0\leq t \leq 1}$ and 
 $Y=(Y(t))_{0\leq t \leq 1}$ denotes a Yule-Simon process with parameter $\rho = 1/p$ 
\end{proposition}
\begin{remark}
The result of Simon \cite{Simon} corresponds to the case when the functional $F$ only depends on the terminal value of the path.  
\end{remark}
We will prove that the stated convergence holds in $L^2(\PP)$ (which, in the present framework, is equivalent to convergence in probability, since  $F$ is bounded), by establishing the so-called propagation of chaos. This requires estimating first and second moments, and 
the calculations rely on the following lemma.

\begin{lemma}\label{L2} For every $n\geq 1$, let $u(n)$ and $v(n)$ denote two independent uniform sample from $[n]\coloneqq \{1,\ldots, n\}$, which are further independent of the sequences $(X_j)_{j\geq 1}$ and $(\varepsilon_i)_{i\geq 2}$. We have as $n\to \infty$:
\begin{enumerate}
\item the conditional distribution of the  time-rescaled counting process $N_{u(n)}(\lfloor \cdot n\rfloor)$
 given $\varepsilon_{u(n)}=0$, 
converges in the sense of Skorohod  towards the law of a Yule-Simon process with parameter $\rho = 1/p$,

\item
the conditional distribution of the pair of time-rescaled counting processes 
$$\left(N_{u(n)}(\lfloor \cdot n\rfloor), N_{v(n)}(\lfloor \cdot n\rfloor)\right)$$
 given $\varepsilon_{u(n)}=\varepsilon_{v(n)}=0$,
converges in the sense of Skorohod  towards the law of a pair of independent Yule-Simon processes with parameter $1/p$.

\end{enumerate}

\end{lemma} 

\begin{proof}  We shall only prove the second assertion, the first being similar and simpler.
Consider two sequences $(j(n))_{n\geq 1}$  and $(k(n))_{n\geq 1}$  such that  $j(n)/n\to u$ and $k(n)/n\to v$, where we first suppose that $0<u<v<1$. We start by fixing $n$ sufficiently large so  $1\leq j(n) <k(n)<n$, 
and work conditionally on $\varepsilon_{j(n)}=\varepsilon_{k(n)}=0$.

The pair of counting processes
${\mathbf N}=(N_{j(n)}(\ell),N_{k(n)}(\ell))_{\ell\geq 0}$ is a time inhomogeneous Markov chain on $\NN^2$ whose transitions can be computed explicitly. Specifically, the path is first deterministic up to time $j(n)$ with
$${\mathbf N}(\ell) =\left\{\begin{matrix} (0,0) &\text{ for }\ell < j(n),\\
(1,0) &\text{ for }\ell = j(n).\end{matrix}\right.
$$
Then for $j(n)\leq \ell < k(n)-1$ and $a$ positive integer,
\begin{eqnarray*}
\PP({\mathbf N}(\ell+1)=(a+1,0) \mid {\mathbf N}(\ell))=(a,0))&= &pa/\ell ,\\
\PP({\mathbf N}(\ell+1))=(a,0) \mid {\mathbf N}(\ell))=(a,0))&= &1-pa/\ell .
\end{eqnarray*}
Since $\varepsilon_{k(n)}=0$, the transition at the $k(n)$-th step is again deterministic with
$${\mathbf N}(k(n))- {\mathbf N}(k(n)-1)=(0,1).$$
Finally, for $\ell\geq k(n)$ and $a,b$ positive integers with $a+b\leq \ell$, one has
\begin{eqnarray*}
\PP({\mathbf N}(\ell+1)=(a+1,b) \mid {\mathbf N}(\ell)=(a,b))&= &pa/\ell ,\\
\PP({\mathbf N}(\ell+1)=(a,b+1) \mid {\mathbf N}(\ell)=(a,b))&= &pb/\ell,\\
\PP({\mathbf N}(\ell+1)=(a,b) \mid {\mathbf N}(\ell)=(a,b))&= &1-p(a+b)/\ell.
\end{eqnarray*}
By a standard result on approximation of Feller processes by discrete time Markov chains (see, e.g. Theorem 19.28 in Kallenberg \cite{Kall}), one readily deduces that  under the conditional law given $\varepsilon_{j(n)}=\varepsilon_{k(n)}=0$, the distribution of the pair of time-rescaled counting processes
$\left( N_{j(n)}(\lfloor \cdot n\rfloor),N_{k(n)}(\lfloor \cdot n\rfloor)\right)$
converges in the sense of Skorohod to that of 
$$\left(\Ind_{\{u\leq t\}} Z(p(\ln t - \ln u)),  \Ind_{\{v\leq t\}} Z'(p(\ln t- \ln v))\right)_{0\leq t \leq 1},$$ where $Z$ and $Z'$  denote two independent Yule process both started from $1$.

That the same holds in the case $0<v<u<1$ is clear by symmetry. 
Since the pair $(u(n)/n, v(n)/n)$ converges in law towards a pair of independent uniform random variables on $[0,1]$ and $\PP(\varepsilon_{u(n)}=0, \varepsilon_{v(n)}=0\mid u(n),v(n))=(1-p)^2$ except when $u(n)=1$, or $v(n)=1$, or $u(n)=v(n)$,
an event whose probability tends to $0$ as $n\to \infty$, our claim now follows readily from the representation of the Yule-Simon process in Lemma \ref{L1}. 
\end{proof}

\begin{proof}[Proof of Proposition \ref{P1}]
Since  $N_j(\lfloor  \cdot n \rfloor)=0$ when $\varepsilon_j=1$ and $F(0)=0$,
we have
 $$ \frac{1}{n}\sum_{j= 1}^{n}F\left( N_j(\lfloor  \cdot n \rfloor)\right) = \frac{1}{n}\sum_{j= 1}^{n} \Ind_{\{\varepsilon_j=0\}}F\left( N_j(\lfloor  \cdot n \rfloor)\right).$$
We deduce from Lemma \ref{L2}(i) that
\begin{eqnarray*}
\lim_{n\to \infty} \EE\left[\frac{1}{n}\sum_{j= 1}^{n}\Ind_{\{\varepsilon_j=0\}} F \left( N_j(\lfloor  \cdot n \rfloor)\right)\right]
&=& \lim_{n\to \infty} \EE\left[\Ind_{\{\varepsilon_{u(n)}=0\}} F \left( N_{u(n)}(\lfloor  \cdot n \rfloor)\right)\right]\\
&=& (1-p) \EE\left[ F(Y)\right] ,
\end{eqnarray*}
Similarly, we deduce from Lemma \ref{L2}(ii) that
\begin{eqnarray*} 
&& \lim_{n\to \infty} \EE\left[\left(\frac{1}{n}\sum_{j= 1}^{n}\Ind_{\{\varepsilon_j=0\}} F \left( N_j(\lfloor  \cdot n \rfloor)\right)\right)^2\right]\\
&=&\lim_{n\to \infty} \EE\left[\frac{1}{n^2}\sum_{j,k= 1}^{n}\Ind_{\{\varepsilon_j=0\}} F \left( N_j(\lfloor  \cdot n \rfloor)\right) \Ind_{\{\varepsilon_k=0\}} F \left( N_k(\lfloor  \cdot n \rfloor)\right)\right]\\
 &=& \lim_{n\to \infty} \EE\left[\Ind_{\{\varepsilon_{u(n)}=0\}} F \left( N_{u(n)}(\lfloor  \cdot n \rfloor)\right)
 \Ind_{\{\varepsilon_{v(n)}=0\}} F \left( N_{v(n)}(\lfloor  \cdot n \rfloor)\right)\right]\\
&=& (1-p)^2 \EE\left[ F(Y)\right]^2,
\end{eqnarray*}
which completes the proof. 
\end{proof}

Our next goal it to extend Proposition \ref{P1} to unbounded  functionals, and this relies on moment bounds for the counting processes $N_j$.
\begin{lemma}\label{L5} For every $\gamma \in(1, 1/p)$, there exist numerical constants $c=c(\gamma,p)>0$
and $\eta=\eta(\gamma,p)\in(0,1)$ such that
$$\EE\left[ N_j(n)^{\gamma}\right] \leq c (n/j)^{\eta}\qquad\text{for all }1\leq j \leq n.$$
\end{lemma}
\begin{proof} We first observe from Jensen's inequality, that working with some $\gamma'\in(\gamma, 1/p)$ in place of $\gamma$, it suffices to establish the inequality in the statement with $\eta=1$.
We shall show that one can pick a constant $c>0$ sufficiently large, such that for every $j\geq 1$, the process 
$(c+N_j(k))^{\gamma}/k$ for $k\geq j$ is a supermartingale. As a consequence, we have 
$$\EE\left[ (c+N_j(n))^{\gamma}\right] \leq (c+1)^{\gamma} n/j,$$
which thus proves our claim. 

The conditional probability that $N_j(k+1)= \ell+1$ given $N_j(k)=\ell$ equals $p\ell/k$ for any $j+\ell\leq k+1$, and therefore 
\begin{eqnarray*} 
& &\EE\left[ \frac{(c+N_j(k+1))^{\gamma}}{k+1} -  \frac{(c+N_j(k))^{\gamma}}{k}\mid N_j(k)=\ell \right ]\\
&=&p \frac{\ell}{k} \frac{(c+\ell +1)^{\gamma}}{k+1} +
\left(1-p \frac{\ell}{k}\right) \frac{(c+\ell)^{\gamma}}{k+1}- \frac{(c+\ell)^{\gamma}}{k}\\
&=& \frac{(c+\ell)^{\gamma}}{k(k+1)}\left[ \ell p \left( (1+1/(c+\ell))^{\gamma}-1\right)-1\right].
\end{eqnarray*}
Since $(1+y)^{\gamma}-1\sim \gamma y$ when $y\to 0+$
and $p \gamma <1$, we can now choose $c$ sufficiently large so that 
$$\ell p \left( (1+1/(c+\ell))^{\gamma}-1\right)< 1 \qquad \text{for all } \ell\geq 1;$$
then the conditional expectation above is non-positive, and the supermartingale property is established.

\end{proof}

\begin{corollary}\label{C4} Suppose that there exists $c>0$ and $1\leq \gamma <1/p$ such that $|F(\omega)|\leq c \omega(1)^{\gamma}$ for every counting function $\omega: [0,1]\to \NN$. 
Then the convergence stated in Proposition \ref{P1} holds in $L^1(\PP)$. 
\end{corollary} 
\begin{proof} Lemma \ref{L5} ensures that for any $\gamma'\in(\gamma,1/p)$, 
\begin{equation}\label{E:suff}
\sup_{n\geq 1}\EE\left[ \frac{1}{n}\sum_{j= 1}^{n}  N_j(n)^{\gamma'} \right]<\infty.
\end{equation}
Our statement then follows from a standard argument of uniform integrability.
More precisely, we may assume without loss of generality that $F\geq 0$, and then set $F_b(\omega)=F(\omega)\wedge b$ for an arbitrary large $b>0$. Next we write 
$$F(\omega)=F_b(\omega) +(F(\omega)-b)^+,$$
and note that since  $F(Y)\leq cY(1)^{\gamma}\in L^1(\PP)$, 
$$\lim_{b\to \infty} \EE[F_b(Y)]=\EE[F(Y)]< \infty.$$
On the other hand, from our assumptions and Markov's inequality, we have 
$$  \EE\left[ \left(F\left( N_j(\lfloor  \cdot n \rfloor)\right)-b\right)^+\right]
\leq c b^{\gamma-\gamma'} \EE[N_j(n)^{\gamma'}],$$
and it follows from \eqref{E:suff} that 
$$\lim_{b\to \infty} \sup_{n\geq 1} \EE\left[ \frac{1}{n}\sum_{j= 1}^{n}\left(F\left( N_j(\lfloor  \cdot n \rfloor)\right)-b\right)^+\right] = 0.$$
Since Proposition \ref{P1} holds for $F_b$, this easily entails our claim. 
\end{proof}

We now have all the ingredients needed for establishing Theorem \ref{T1}

\begin{proof} For every $n\geq 1$, we use the notation with an exponent $(n)$ for quantities related to the skeleton
random walk $S^{(n)}$. In particular, $X^{(n)}_j=\xi(j/n)- \xi((j-1)/n)$ denotes its $j$-th step,   and $N^{(n)}_j$ refers to the  $j$-th counting process arising in the step reinforcement for $S^{(n)}$.

Fix $k\geq 1$, $\theta_1, \ldots, \theta_k\in\RR$ and  $t_1, \ldots, t_k\in[0,1]$, and recall from \eqref{E:SwithN} that 
$$ \theta_j \cdot \hat S^{(n)}(\lfloor nt_j\rfloor) = \sum_{\ell =1}^n N_{\ell}^{(n)}(\lfloor nt_j\rfloor)  \theta_j \cdot X^{(n)}_{\ell} .
$$
Recalling from \eqref{E:LKF} that $\Psi$ denotes the characteristic exponent of $\xi$, we now see by computing first the expectation with respect to the i.i.d. steps that
$$\EE\left[\exp\left\{ \iu \sum_{j=1}^k \theta_j \cdot \hat S^{(n)}(\lfloor nt_j\rfloor)
\right\} \right] = \EE\left[\exp\left\{-\frac{1}{n}\sum_{\ell =1}^n \Psi\left( \sum_{j=1}^k N_{\ell}^{(n)}(\lfloor nt_j\rfloor)  \theta_j \right) \right\} \right].
$$

This yields us to considering the functional
$$F(\omega) =  \Psi\left( \sum_{j=1}^k \omega(\lfloor nt_j\rfloor)  \theta_j\right),$$
where $\omega: [0,1] \to \NN$ stands for a generic counting function. 
Recall from Lemma 3.1 in \cite{BG61} that for every $\eta>0$,  one has as $|z|\to \infty$
$$| \Psi(z)|=\left\{ \begin{matrix} &o(|z|^{2+\eta})& \text{ when }{\mathbf q}\not=0, \\
&o(|z|^{\beta(\Lambda)+\eta})& \text{ when }{\mathbf q}=0\text{ and }\int_{|x|\leq 1}|x|\Lambda (\dd x)=\infty ,\\
&o(|z|^{1+\eta})&\text{ when }{\mathbf q}=0\text{ and }\int_{|x|\leq 1}|x|\Lambda (\dd x)<\infty .\\
\end{matrix} \right.$$  
Further $\Psi(0)=0$, and observing that 
$$\left | \sum_{j=1}^k \omega(t_j)  \theta_j\right | \leq \omega(1) \sum_{j=1}^k |\theta_j|,$$
we see that the assumption that the memory parameter $p$ is admissible ensures that the requirement of Corollary \ref{C4} is fulfilled with any $\gamma \in(1, 1/p)$. 
We conclude that
$$\lim_{n\to \infty} \EE\left[ \exp\left\{ \iu \sum_{j=1}^k \theta_j \cdot \hat S^{(n)}(\lfloor nt_j\rfloor)
\right\} \right] = \exp\left\{ -(1-p) \EE
\left[ {\Psi}\left(\sum_{j=1}^k Y(t_j)  \theta_j\right)\right]\right\}.$$ 
Comparing with Corollary \ref{C3}, this completes the proof.
\end{proof} 

We now conclude this work by briefly discussing the situation when the memory parameter is not admissible for $({\mathbf q}, a, \Lambda)$.  In this direction, recall that Blumenthal and Getoor \cite{BG61} also introduced a so-called {\em lower index} of a L\'evy process $\xi$ with characteristic exponent $\Psi$ as
$$\beta''=\sup\{\alpha\geq 0: \lim_{|\theta|\to \infty} |\theta|^{-\alpha} \Re\Psi(\theta)=\infty \}.$$
There is always the inequality  $\beta''\leq \beta$, and equality may hold (for instance in the stable case).
Now if we assume that $p>1/\beta''$, then we have $\EE[\Re\Psi(Y(1))]=\infty$, and 
 it follows from Proposition \ref{P1} and a classical truncation argument that for every $\theta\in\RR\backslash \{0\}$
$$\lim_{n\to \infty} \frac{1}{n}\sum_{j= 1}^{n} \Re\Psi\left( N_j(n) \theta  \right)= \infty,\qquad \text{in probability}.$$
The calculation in the proof of Theorem \ref{T1} now shows that 
$$\lim_{n\to \infty}\EE\left[\exp\left\{ \iu \theta \cdot \hat S^{(n)}(n)
\right\} \right] = 0,
$$
and hence the sequence of variables $\hat S^{(n)}(n)$ cannot converge weakly. 

When $\xi$ is an isotropic $\alpha$-stable L\'evy process for some $\alpha\in(0,2]$, it follows from Theorem 1 of Businger \cite{Buart}  that the process $n^{1/\alpha -p} \hat S^{(n)}(\lfloor tn\rfloor )$ converges in distribution to 
$t^p \eta$ as $n\to \infty$, where $\eta$ is some non-degenerate random vector in $\RR$. 
A perusal of her arguments shows that this feature remains provided that $\Psi(c\theta)\sim c^{\alpha}\Psi_{\alpha} (\theta)$ as $c\to \infty$, where $\Psi_{\alpha}$ is the characteristic exponent of an $\alpha$-stable L\'evy process. Nonetheless, it is also seen from her approach that one cannot expect a similar result to hold when $\xi$ is an arbitrary  L\'evy process, even if we require its upper and lower Blumenthal-Getoor indices $\beta$ and $\beta''$ to coincide with $\alpha$.

\bibliography{NoiseR.bib}

\end{document}